\documentclass[12pt]{article}

\usepackage{epsfig}
\usepackage{latexsym}
\usepackage{amsmath}
\usepackage{amsfonts}
\usepackage{amssymb}
\usepackage{graphicx}%
\usepackage{algorithmic}
\usepackage{color}

\makeatletter

\newcommand{\Rmnum}[1]{\expandafter\@slowromancap\romannumeral #1@}
\makeatother


\begin{document}
	
	\pagestyle{myheadings} \markright{\sc Packing of spanning mixed arborescences\hfill} \thispagestyle{empty}
	
	\newtheorem{theorem}{Theorem}[section]
	\newtheorem{corollary}[theorem]{Corollary}
	\newtheorem{definition}[theorem]{Definition}
	\newtheorem{guess}[theorem]{Conjecture}
	\newtheorem{claim}[theorem]{Claim}
	\newtheorem{problem}[theorem]{Problem}
	\newtheorem{question}[theorem]{Question}
	\newtheorem{lemma}[theorem]{Lemma}
	\newtheorem{proposition}[theorem]{Proposition}
	\newtheorem{fact}[theorem]{Fact}
	\newtheorem{acknowledgement}[theorem]{Acknowledgement}
	\newtheorem{algorithm}[theorem]{Algorithm}
	\newtheorem{axiom}[theorem]{Axiom}
	\newtheorem{case}[theorem]{Case}
	\newtheorem{conclusion}[theorem]{Conclusion}
	\newtheorem{condition}[theorem]{Condition}
	\newtheorem{conjecture}[theorem]{Conjecture}
	\newtheorem{criterion}[theorem]{Criterion}
	\newtheorem{example}[theorem]{Example}
	\newtheorem{exercise}[theorem]{Exercise}
	\newtheorem{notation}[theorem]{Notation}
	\newtheorem{observation}[theorem]{Observation}
	\newtheorem{solution}[theorem]{Solution}
	\newtheorem{summary}[theorem]{Summary}
	
	\newtheorem{thm}[theorem]{Theorem}
	\newtheorem{prop}[theorem]{Proposition}
	\newtheorem{defn}[theorem]{Definition}

	\newtheorem{lem}[theorem]{Lemma}
	\newtheorem{con}[theorem]{Conjecture}
	\newtheorem{cor}[theorem]{Corollary}

	\newenvironment{proof}{\noindent {\bf
			Proof.}}{\rule{3mm}{3mm}\par\medskip}
	\newcommand{\remark}{\medskip\par\noindent {\bf Remark.~~}}
	\newcommand{\pp}{{\it p.}}
	\newcommand{\de}{\em}

	\newcommand{\g}{\mathrm{g}}

	\newcommand{\qf}{Q({\cal F},s)}
	\newcommand{\qff}{Q({\cal F}',s)}
	\newcommand{\qfff}{Q({\cal F}'',s)}
	\newcommand{\f}{{\cal F}}
	\newcommand{\ff}{{\cal F}'}
	\newcommand{\fff}{{\cal F}''}
	\newcommand{\fs}{{\cal F},s}
	\newcommand{\cs}{\chi'_s(G)}
	
	\newcommand{\G}{\Gamma}
	\newcommand{\wrt}{with respect to }
	\newcommand{\mad}{{\rm mad}}
	\newcommand{\col}{{\rm col}}
	\newcommand{\gcol}{{\rm gcol}}
	
	\newcommand*{\ch}{{\rm ch}}
	\newcommand*{\ra}{{\rm ran}}
	\newcommand{\co}{{\rm col}}
	\newcommand{\sco}{{\rm scol}}
	\newcommand{\wc}{{\rm wcol}}
	\newcommand{\dc}{{\rm dcol}}
	\newcommand*{\ar}{{\rm arb}}
	\newcommand*{\ma}{{\rm mad}}
	\newcommand{\di}{{\rm dist}}
	\newcommand{\tw}{{\rm tw}}
	\newcommand{\scol}{{\rm scol}}
	\newcommand{\wcol}{{\rm wcol}}
	\newcommand{\td}{{\rm td}}
	\newcommand{\edp}[2]{#1^{[\natural #2]}}
	\newcommand{\epp}[2]{#1^{\natural #2}}
	\newcommand*{\ind}{{\rm ind}}
	\newcommand{\red}[1]{\textcolor{red}{#1}}
	
	\def\C#1{|#1|}
	\def\E#1{|E(#1)|}
	\def\V#1{|V(#1)|}
	\def\iarb{\Upsilon}
	\def\ipac{\nu}
	\def\nul{\varnothing}

	\newcommand*{\QEDA}{\ensuremath{\blacksquare}}
	\newcommand*{\QEDB}{\hfill\ensuremath{\square}}

	\title{\Large\bf  Packing of spanning mixed arborescences}
	
	\author{Hui Gao \\
		Center for Discrete Mathematics\\
		Fuzhou University\\
		Fuzhou, Fujian 350108, China\\
		E-mail: \texttt{gaoh1118@yeah.net}  
		\and
		Daqing Yang\thanks{Corresponding author,  grant number:	NSFC  11871439.} \\
		Department of Mathematics \\
		Zhejiang Normal University \\
		Jinhua, Zhejiang 321004, China\\
		E-mail: \texttt{dyang@zjnu.edu.cn} 
	}

	\maketitle 

\begin{abstract}
	In this paper, we characterize a mixed graph $F$ which contains 
	 $k$ edge and arc disjoint spanning mixed arborescences $F_{1}, \ldots, F_{k}$, such that for each $v \in V(F)$, the cardinality of $\{i \in [k]: v \text{ is the root of } F_{i}\}$ lies in some prescribed interval. 
	 This generalizes both Nash-Williams and Tutte's theorem on spanning tree packing for undirected graphs  and the previous characterization on digraphs which was given by 
	 Cai [in: Arc-disjoint arborescences of digraphs, J. Graph Theory 7(2) (1983), 235-240] and Frank [in: On disjoint trees and arborescences,  Algebraic Methods in Graph Theory, Colloquia Mathematica Soc. J. Bolyai, Vol. 25 (North-Holland, Amsterdam) (1978), 159-169].   
\end{abstract}



\section{Introduction} 
In this paper, we consider graphs which may have multiple edges or (and) arcs but not loops. 
\emph{A mixed graph} $F = (V ; E ,A)$ is a graph consisting of the set $E$ of undirected edges and the set $A$ of directed arcs. 
Let $X_{1}, \ldots, X_{t}$ be disjoint subsets of $V$, we call $\mathcal{P}= \{X_{1}, \ldots, X_{t} \}$ a {\em subpartition of $V$} and particularly a {\em partition of  $V$} if $V=\cup_{j=1}^{t}X_{j}$. 
Denote $e_{E}(\mathcal{P})= |\{e \in E: $ one end of $e$ belongs to some $X_{i}$ and the other end belongs to another $X_{j}$ with $j \neq i$ or $V \setminus \cup_{j=1}^{t}X_{j}$  $ \}|$.   
Denote the set $\{1,\ldots,k\}$ by $[k]$.  
For a function $f: V \rightarrow \mathbb{N}$, define a set function $\widetilde{f}: 2^{V} \rightarrow \mathbb{N}$ as $\widetilde{f}(X)=\sum_{x \in X} f(x)$, where $X \subseteq V$.  

Nash-Williams~\cite{NW61} and Tutte~\cite{T} independently characterized
when an undirected graph has $k$ edge-disjoint spanning trees. 


\begin{thm} [\cite{NW61,T}]  \label{packingtrees} 
For a graph $G=(V,E)$, there exist $k$ edge disjoint spanning trees, if and only if for any partition $\mathcal{P} = \{X_{0}, X_{1}, \ldots, X_{t} \}$  of $V$, 
\[
e_{E}(\mathcal{P}) \geq kt.
\]
\end{thm}

Let $D=(V,A)$ be a digraph.  A subdigraph of $D$ is  spanning if its vertex set is $V$.  A subdigraph $F$ (it may not be spanning) of $D$ is an {\em $r$-arborescence} if its underlying graph is a tree and for any $u \in V(F)$, there is exactly one directed path in $F$ from $r$ to $u$,  vertex $r$ is the root of arborescence $F$.

As a directed version of Theorem \ref{packingtrees}, Edmonds' theorem \cite{edmonds} characterizes directed graphs that contain $k$ arc disjoint spanning arborescences with prescribed roots in terms of a cut condition. 


\begin{thm}[\cite{edmonds}]\label{spanningarb}
For a digraph $D = (V,A)$, let $R = \{r_{1}, \ldots, r_{k} \} \subseteq V$ be a multiset. 
For $i = 1,\ldots,k$, there exist arc-disjoint spanning $r_{i}$-arborescneces in $D$,  if and only if for any $ \emptyset \neq X \subseteq V$,
\[
d_{A}^{-}(X) \geq |\{r_{i}: r_{i} \notin X \}|. 
\]
\end{thm}

Throughout this paper, $F = (V ; E ,A)$ is a mixed graph,  $R=\{r_{1}, \ldots, r_{k}\} \subseteq V$ is a multiset. 
By regarding each undirected edge as a directed arc in both directions, each concept in directed graphs can be naturally extended to mixed graphs. Especially, a subdigraph $P$ of $F$ is a {\em mixed path} if its underlying graph is a path and one end of $P$ can be reached from the other. 
A subdigraph $T$ (it may not be spanning) of $F$ is called an \emph{$r$-mixed arborescence} if its underlying graph is a tree and for any $u \in V(T)$, there is exactly one mixed path in $T$ from $r$ to $u$. Equivalently, a subgraph $T$ of $F$ is an $r$-mixed arborescence if there exists an orientation of the undirected edges of $T$ such that the obtained subgraph (whose arc set is the union of original arc set and oriented arc set of $T$) is an $r$-arborescence.

The following result is due to Frank \cite{F-13}, it generalized Theorem~\ref{packingtrees} and Theorem~\ref{spanningarb} to mixed graphs when $r_{1}=r_{2}= \dots = r_{k}$. 
Since mixed graphs may contain edges and arcs, a property that holds for mixed graphs should hold for both directed and undirected graphs.

\begin{thm}[\cite{F-13}]
Let $F = (V; E,A)$ be a mixed graph, $r \in V$, and $k$ a positive integer.  
There exist $k$ edge and arc disjoint spanning $r$-mixed arborescences in $F$, if and only if,
for any subpartition $\{ X_{1}, \ldots, X_{t} \}$ of $V-r$,
\[
e_{E}(\mathcal{P}) + \sum_{j=1}^{t}d_{A}^{-}(X_{j}) \geq kt.
\]
\end{thm}

For nonempty $X,Z \subseteq V$,  let $Z \rightarrow X$ denote that $X$ and $Z$ are disjoint and $X$ is {\em reachable} from $Z$, that is, there is a mixed path from $Z$ to $X$.
We shall  write $v$ for $\{v\}$ for simplicity.
Let $P(X):= X \cup \{v \in V \setminus X : v \rightarrow X \}$.

Let $D$ and $R$ be as in Theorem~\ref{spanningarb}. 
The following remarkable extension of Edmonds' theorem (by Kamiyama, Katoh and Takizawa \cite{kamiyama}) enables us to find a 
{\em packing of reachable arborescences in digraph $D$ for $R$}, that is packing  arc disjoint $r_{i}$-arborescences $F_{i}$ in $D$ for $1 \leq i \leq k$ such that $V(F_{i})=\{v \in V:$ $v$ is reachable from $r_{i}$ in $D \}$.

\begin{theorem}(\cite{kamiyama}) \label{19}
	In a digraph $D = (V,A)$, let $R= \{r_{1}, \ldots, r_{k}\} \subseteq V$ be a multiset. There exists a packing of reachable arborescences in digraph $D$ for $R$ 
	if and only if for any $\emptyset \neq X \subseteq V$,
	\[
	d_{A}^{-}(X) \geq |\{ r_{i}: r_{i} \in P(X) \setminus X  \} |.
	\]
\end{theorem} 

For packing of reachable arborescences, further extensions have been made, such as its mixed version \cite{matsuoka}, matroidal version \cite{kiraly}, and matroidal mixed version \cite{GY-max-mix}.  
Some other developments for packing of arborescences in the recent years include matroid-based packing \cite{durand}, and its hypergraphic version \cite{fortier},  under cardinality constraints \cite{GY-1}.   Refer \cite{fortier} for more details.

In this paper, we are interested in the following extension of Edmonds' theorem, 
which is due to Cai \cite{C-6} and Frank \cite{F-13} (see also Theorem 10.1.11 in Frank \cite{F-2011}).  
This extension characterized a digraph $D$ which contains $k$ arc disjoint spanning arborescences $F_{1}, \ldots, F_{k}$, such that for each $v \in V(D)$, the cardinality of $\{i \in [k]: v \text{ is the root of } F_{i}\}$ lies in some prescribed interval.

\begin{thm} [\cite{C-6,F-13}] \label{Cai-21}
Let $D=(V,A)$ be a digraph, $f, g: V \rightarrow \mathbb{N}$ be functions such that $f \leq g$. Then there exist $k$ arc-disjoint spanning arborescences $F_{1}, \ldots, F_{k}$ in $D$ for which $F_{i}$ is rooted at some $r_{i} \in V$  for $1 \leq i \leq k$ such that $f(v) \leq  |\{i \in [k]: r_{i}=v \}| \leq g(v)$ for $v \in V$, if
and only if,
\begin{itemize}
\item[(i)] $\widetilde{f}(V) \le k$;
\item[(ii)] for any subpartition $\{ X_{1}, \ldots, X_{t} \}$ of $V$,
\begin{equation} \label{C-F-1}
\sum^{t}_{j=1}d_{A}^{-} (X_{j}) \geq k(t-1) + \widetilde{f}(V \setminus \cup_{j=1}^{t}X_{j});
\end{equation}
\item[(iii)] for any $ \emptyset \neq X \subseteq V$,
\begin{equation}\label{dir-ine-g}
d_{A}^{-} (X)  \ge k -  \widetilde{g}(X).
\end{equation}
\end{itemize}
\end{thm}


In this paper, we generalize  Theorem \ref{packingtrees}  
and  Theorem \ref{Cai-21} to mixed graphs, which is the following.

\begin{thm} \label{mix-spanning}
Let $F=(V; E, A)$ be a mixed graph, $f, g: V \rightarrow \mathbb{N}$ be functions such that $f \leq g$. Then there exist $k$ edge and arc disjoint spanning mixed arborescences $F_{1}, \ldots, F_{k}$ in $F$ for which $F_{i}$ is rooted at some $r_{i} \in V$  for $1 \leq i \leq k$ and $f(v) \leq  |\{i \in [k]: r_{i}=v \}| \leq g(v)$ for $v \in V$, if and only if,
\begin{itemize}
\item[(i)] $\widetilde{f}(V) \le k$;
\item[(ii)] for any subpartition $\mathcal{P}=\{X_{1}, \ldots, X_{t} \}$ of $V$,
\begin{equation} \label{ine-f}
e_{E}(\mathcal{P})+\sum^{t}_{j=1}d_{A}^{-} (X_{j}) \geq k(t-1) + \widetilde{f}(V \setminus \cup_{j=1}^{t}X_{j});
\end{equation}
\item[(iii)] for any subpartition $\mathcal{P}=\{X_{1}, \ldots, X_{t} \}$ of $V$,
\begin{equation}\label{ine-g}
e_{E}(\mathcal{P})+\sum^{t}_{j=1}d_{A}^{-} (X_{j}) \geq kt - \widetilde{g}(\cup_{j=1}^{t}X_{j}).
\end{equation}
\end{itemize} 
\end{thm} 

For the proof  of our result, besides what have been used by Nash-Williams~\cite{NW61}, Tutte~\cite{T}, Cai \cite{C-6} and Frank \cite{F-13}, we adopt a technique named as {\em properly intersecting elimination operation (PIEO for simplicity)}, which was first introduced by B\'{e}rczi and Frank \cite{berczi-frank1} (to the best of our knowledge), studied and used again by Gao and Yang \cite{GY-1}. 
Indeed, we shall use some similar approaches to \cite{GY-1} in our proofs of Theorem~\ref{mix-spanning}.

\section{Proof of Theorem~\ref{mix-spanning}}

We shall use some definitions  
and propositions that have been presented in \cite{GY-1}.    

Let $\Omega$ be a finite set.  
Two subsets $X, Y \subseteq \Omega$ are {\em intersecting} if $X \cap Y \neq \emptyset$ and {\em properly intersecting} if $X \cap Y$, $X \setminus Y$, and $Y \setminus X \neq \emptyset$. 
A function $p: 2^{\Omega} \rightarrow \mathbb{Z}$ is {\em supermodular (intersecting supermodular)}, where $2^{\Omega}$ denotes the power set of $\Omega$, if the inequality
\[
p(X)+p(Y) \leq p(X \cup Y)+ p(X \cap Y)
\]
holds for all subsets (intersecting subsets, respectively) of $\Omega$. A function $b$ is {\em submodular} if $-b$ is supermodular.

Let $\mathcal{F}$ be a multiset, which consists of some subsets of $\Omega$ (these subsets do not have to be different). Let $\cup \mathcal{F}$ be the union of elements in $\mathcal{F}$ (then $\cup \mathcal{F} \subseteq \Omega$).
Let $x \in \Omega $ and $\mathcal{F}(x) $ denote the number of elements  in $\mathcal{F}$ containing $x$. If there exist no properly intersecting pairs in $\mathcal{F}$, then $\mathcal{F}$ is {\em laminar}. If there exists a properly intersecting pair $X$ and $Y$ in $\mathcal{F}$, then we obtain $\mathcal{F}'$ from $\mathcal{F}$ by replacing $X$ and $Y $ with one of the following three types of subset(s):

\noindent {\em Type $1$},
$X \cup Y$ and $X \cap Y$, denoted as $\mathcal{F} \xrightarrow{1} \mathcal{F}'$;

\noindent {\em Type $2$},  $X \cup Y$, denoted as $\mathcal{F} \xrightarrow{2} \mathcal{F}'$;

\noindent {\em Type $3$}, $X \cap Y$, denoted as $\mathcal{F} \xrightarrow{3} \mathcal{F}'$.

\noindent{\em A properly intersecting elimination operation (PIEO for simplicity) on $X $ and $ Y$ in $\mathcal{F}$} is defined to be one of the above three types.

Let $Z_{1}$ and $Z_{2}$ be multisets. Denote by $Z_{1} \uplus Z_{2}$ the {\em multiset union} of $Z_{1}$ and $Z_{2}$, that is, for any $z$, the number of $z$ in $Z_{1} \uplus Z_{2}$ is the total number of $z$ in $Z_{1}$ and $Z_{2}$.


Let {\em $\mathcal{D}(\Omega)$ be the set that consists of all families of disjoint subsets of $\Omega$}.  
From now on, we suppose $\mathcal{F}_{1}, \mathcal{F}_{2} \in \mathcal{D}(\Omega)$. 
We adopt PIEOs in $\mathcal{G}_{0}=\mathcal{F}_{1} \uplus \mathcal{F}_{2}$, step by step, and obtain families   $\mathcal{G}_{0}, \ldots, \mathcal{G}_{i-1},\mathcal{G}_{i},\ldots$ of subsets  of $\Omega$. 
In \cite{GY-1}, it has already been proved that the process of PIEOs will terminate.
Suppose the obtained  families of subsets  of $\Omega$ are
$\mathcal{G}_{0}, \ldots, \mathcal{G}_{n}$. Then  $\mathcal{G}_{n}$ is laminar. 
For $X, Y \in \mathcal{G}_{i}$, if  $X \subseteq Y$, then define $X \le Y$.   
Let $\mathcal{G}_{i}'$ be the family of maximal elements in $\mathcal{G}_{i}$, $\mathcal{F}_{3}:= \mathcal{G}_{n}' $ and $\mathcal{F}_{4}:=\mathcal{G}_{n} \setminus \mathcal{F}_{3} $.

\begin{proposition}{\rm(\cite[Proposition 3.2]{GY-1})}\label{47-P}
	If $X, Y  \in \mathcal{G}_{i} $ are properly intersecting, then $X, Y \in \mathcal{G}_{i}'$.
\end{proposition}

\begin{proposition}{\rm(\cite[Proposition 3.4]{GY-1})}\label{17-P}
	The following hold true:
\begin{enumerate}
	\item[(i)] 	$\mathcal{F}_{3}, \mathcal{F}_{4} \in \mathcal{D}(\Omega)$. $\cup\mathcal{F}_{4} \subseteq (\cup \mathcal{F}_{1}) \cap (\cup \mathcal{F}_{2})$.
	
	\item[(ii)]
	 Moreover, $\cup\mathcal{F}_{4} = (\cup \mathcal{F}_{1}) \cap (\cup \mathcal{F}_{2})$ if and only if for any $i \in [n]$, $\mathcal{G}_{i-1} \xrightarrow{1} \mathcal{G}_{i}$.
\end{enumerate}
\end{proposition}

Now we are ready for the \textbf{proof of Theorem~\ref{mix-spanning}}.

\smallskip

\noindent\textbf{($\Rightarrow$) Necessity:} 
Suppose there exist $k$ edge and arc disjoint spanning mixed arborescences $F_{1}, \ldots, F_{k}$ in $F$ for which $F_{i}$ is rooted at some $r_{i} \in V$  for $1 \leq i \leq k$ and $f(v) \leq  |\{i \in [k]: r_{i}=v \}| \leq g(v)$ for $v \in V$,

Then there exists an oriented arc set $A'$ of $E$, for which there exist $k$ arc-disjoint spanning arborescences $F'_{1}, \ldots, F'_{k}$ in $D=(V, A\cup A' )$ such that $F'_{i}$ is rooted at $r_{i}$  for $1 \leq i \leq k$ and $f(v) \leq  |\{i \in [k]: r_{i}=v \}| \leq g(v)$ for $v \in V$. Obviously, $\widetilde{f}(V) \leq k$. 
By Theorem~\ref{Cai-21}, (\ref{C-F-1}) and (\ref{dir-ine-g}) hold in $D$.  Let $\mathcal{P}=\{X_{0}, X_{1}, \ldots, X_{t} \}$ be a partition of $V$. Since $A'$ is an oriented arc set of $E$, $e_{E}(\mathcal{P}) \geq \sum_{j=1}^{t}d_{A'}^{-}(X_{j})$. Hence, by (\ref{C-F-1}),
\[
e_{E}(\mathcal{P}) +\sum_{j=1}^{t}d_{A}^{-}(X_{j}) \geq \sum_{j=1}^{t}d_{A'}^{-}(X_{j})+ \sum_{j=1}^{t}d_{A}^{-}(X_{j})=\sum_{j=1}^{t}d_{A \cup A'}^{-}(X_{j}) \geq k(t-1)+ \widetilde{f}(X_{0}),
\]
this is (\ref{ine-f}).  By (\ref{dir-ine-g}),
\[
e_{E}(\mathcal{P}) +\sum_{j=1}^{t}d_{A}^{-}(X_{j}) \geq \sum_{j=1}^{t}d_{A \cup A'}^{-}(X_{j}) \geq \sum_{j=1}^{t}(k - \widetilde{g}(X_{j}))=kt- \widetilde{g}(\cup_{j=1}^{t}X_{j}),
\]
this is (\ref{ine-g}).

\smallskip 

\noindent\textbf{($\Leftarrow$) Sufficiency:}  
We prove the sufficiency by induction on $|E|$. 
For the base step, suppose $E=\emptyset$; and  then apply  Theorem~\ref{Cai-21}. 

For the induction step, suppose $E \neq \emptyset$.  
We shall prove that we can orient an edge $e \in E$ to $\vec{e}$, 
such that after we do  $A:= A+\vec{e}$, $E:=E-e$, $F':=(V; A, E)$,  assumptions (\ref{ine-f}) and (\ref{ine-g}) still hold for the new mixed graph $F'$. 
Then by the induction hypothesis, there exist $k$ edge and arc disjoint mixed arborescences $F_{1}, \ldots, F_{k}$ in $F'$ for which $F_{i}$ is rooted at some $r_{i} \in V$  for $1 \leq i \leq k$ such that $f(v) \leq  |\{i \in [k]: r_{i}=v \}| \leq g(v)$ for $v \in V$. If $\vec{e} \notin \cup_{i=1}^{k}F_{i}$, then $F$ includes $F_{1}, \ldots, F_{k}$ as demanded; if $\vec{e} \in E(F_{i_{0}})$ for some $i_{0} \in [k]$, then $F$ includes $F_{1}, \ldots, F_{i_{0}}-\vec{e}+e, \ldots, F_{k}$ as demanded. 
 
The critical point that determines an orientation of $e$ lies on the subpartitions of $V$ that make  assumptions (\ref{ine-f}) or  (\ref{ine-g}) tight in $F$. These critical subpartitions are defined next;   $\mathcal{E}^{1}$  is aimed at  (\ref{ine-f}),  $\mathcal{E}^{2}$ is aimed  at  (\ref{ine-g}). This explains why the subpartitions in $\mathcal{E}^{1}$ and $\mathcal{E}^{2}$ will play some central roles next.    
 
Define
\[
\begin{split}
\mathcal{E}^{1} &:=\{ \mathcal{F} \in \mathcal{D}(V) : e_{E}(\mathcal{F})+ \sum_{X \in \mathcal{F}}d_{A}^{-}(X)=k(t-1)+\widetilde{f}(V \setminus \cup \mathcal{F}) \};\\
\mathcal{E}^{2} &:=\{ \mathcal{F} \in \mathcal{D}(V): e_{E}(\mathcal{F})+ \sum_{X \in \mathcal{F}}d_{A}^{-}(X)=kt-\widetilde{g}( \cup \mathcal{F}) \}.
\end{split} 
\] 
Suppose $\mathcal{F}_{1}, \mathcal{F}_{2} \in \mathcal{E}^{1} \cup \mathcal{E}^{2}$, denote 
\[
E(\mathcal{F}_{1}, \mathcal{F}_{2}):=\{e \in E: \text{one end of $e$ is in $\cup \mathcal{F}_{1} \setminus \cup \mathcal{F}_{2}$  and the other is in $\cup \mathcal{F}_{2} \setminus \cup \mathcal{F}_{1}$}  \}.
\]
\noindent \textbf{Process of PIEOs.}  Let  $\mathcal{G}_{0}=\mathcal{F}_{1} \uplus \mathcal{F}_{2}$. We adopt PIEOs of Type $1$ in $\mathcal{G}_{0}=\mathcal{F}_{1} \uplus \mathcal{F}_{2}$, step by step, and obtain families   $\mathcal{G}_{0}, \ldots, \mathcal{G}_{n}$ of subsets  of $V$; equivalently, for $0 \leq i \leq n-1$, we replace some properly intersecting pair $X$ and $Y$ in $\mathcal{G}_{i-1}$ with $X \cup Y$ and $X \cap Y$, and obtain $\mathcal{G}_{i}$. 
Recall that $\mathcal{G}_{i}'$ is the family of maximal elements in  $\mathcal{G}_{i}$, $\mathcal{F}_{3}:= \mathcal{G}_{n}' $ and $\mathcal{F}_{4}:=\mathcal{G}_{n} \setminus \mathcal{F}_{3} $. By Proposition~\ref{17-P}, $\mathcal{F}_{3}, \mathcal{F}_{4} \in \mathcal{D}(V)$. 

\begin{claim} \label{fact}
\begin{enumerate}
\item[(i)] $|\mathcal{F}_{1}|+|\mathcal{F}_{2}|=|\mathcal{F}_{3}|+|\mathcal{F}_{4}|$.
\item[(ii)]  $\cup \mathcal{F}_{3}=(\cup \mathcal{F}_{1}) \cup (\cup \mathcal{F}_{2})$.
\end{enumerate}
\end{claim}

\begin{proof}  
For $i \in [n]$, suppose we replace a properly intersecting pair $X$ and $Y$ in $\mathcal{G}_{i-1}$ with $ X \cup Y$ and $X \cap Y$, and obtain $\mathcal{G}_{i}$. Clearly, $|\mathcal{G}_{i-1}|=|\mathcal{G}_{i}|$.  
It follows that $|\mathcal{F}_{1}|+|\mathcal{F}_{2}|=  |\mathcal{G}_{0}|=|\mathcal{G}_{n}|=|\mathcal{F}_{3}|+|\mathcal{F}_{4}|$.

By Proposition~\ref{47-P}, $X, Y \in \mathcal{G}_{i-1}'$, and thus $X \cup Y \in \mathcal{G}_{i}'$. So $\mathcal{G}_{i}'$ consists of $X \cup Y$ and all the subsets in $\mathcal{G}_{i-1}'$ not contained in $X \cup Y$; this proves $ \cup \mathcal{G}'_{i-1} = \cup \mathcal{G}'_{i}$. 
Hence $\cup \mathcal{G}'_{0}= \cup \mathcal{G}'_n$.   

By definition, $(\cup \mathcal{F}_{1}) \cup (\cup \mathcal{F}_{2})= \cup \mathcal{G}'_{0}$; since $\cup \mathcal{G}'_{0}=\cup \mathcal{G}'_{n} $ and $\mathcal{G}'_{n}=\mathcal{F}_{3}$, we have $ (\cup \mathcal{F}_{1}) \cup (\cup \mathcal{F}_{2})= \cup \mathcal{G}'_{0}= \cup \mathcal{G}'_{n}=  \cup \mathcal{F}_{3}$. 
\end{proof} 

\begin{claim} \label{>=|E|}
For $\mathcal{F}_{1}, \mathcal{F}_{2} \in \mathcal{E}^{1} \cup \mathcal{E}^{2}$, we have
\begin{align*}
 & e_{E}(\mathcal{F}_{1})+ \sum_{X \in \mathcal{F}_{1}}d_{A}^{-}(X) + e_{E}(\mathcal{F}_{2})+ \sum_{X \in \mathcal{F}_{2}}d_{A}^{-}(X) \\
 \geq ~ & e_{E}(\mathcal{F}_{3})+ \sum_{X \in \mathcal{F}_{3}}d_{A}^{-}(X)+ e_{E}(\mathcal{F}_{4})+ \sum_{X \in \mathcal{F}_{4}}d_{A}^{-}(X) + |E(\mathcal{F}_{1}, \mathcal{F}_{2})|.
\end{align*}
\end{claim}

\begin{proof}
Define  $A''$ is an orientation of  $E$ as following:
\begin{itemize}
	\item if $e=uv \in E$ satisfied that $u \notin \cup \mathcal{F}_{1}$ and $v \in \cup \mathcal{F}_{1}$, orient $e$ from $u$ to $v$ in $A''$; 
	\item else if $e=uv \in E$ satisfied that $u \notin  \cup \mathcal{F}_{2} $ and $v \in \cup \mathcal{F}_{2}$, orient $e$ from $u$ to $v$ in $A''$; 
	\item else, orient the rest of $E$ arbitrarily in $A''$. 
\end{itemize} 
Then $e_{E}(\mathcal{F}_{1})= \sum_{X \in \mathcal{F}_{1}} d_{A''}^{-}(X)$, and $e_{E}(\mathcal{F}_{2})= $ $\sum_{X \in \mathcal{F}_{2}} $ $d_{A''}^{-}(X) $ $+ |E(\mathcal{F}_{1}, \mathcal{F}_{2})|$. 
Hence,
\begin{equation} \label{left}
\begin{split}
& e_{E}(\mathcal{F}_{1})+ \sum_{X \in \mathcal{F}_{1}}d_{A}^{-}(X) + e_{E}(\mathcal{F}_{2})+ \sum_{X \in \mathcal{F}_{2}}d_{A}^{-}(X)\\
 = ~ & \sum_{X \in \mathcal{F}_{1}} d_{A''}^{-}(X) + \sum_{X \in \mathcal{F}_{1}}d_{A}^{-}(X) + \sum_{X \in \mathcal{F}_{2}} d_{A''}^{-}(X) + |E(\mathcal{F}_{1}, \mathcal{F}_{2})| +\sum_{X \in \mathcal{F}_{2}}  d_{A}^{-}(X) \\
 = ~ & \sum_{X \in \mathcal{F}_{1}} d_{A \cup A''}^{-}(X)+ \sum_{X \in \mathcal{F}_{2}} d_{A \cup A''}^{-}(X) + |E(\mathcal{F}_{1}, \mathcal{F}_{2})| \\
 = ~ & \sum_{X \in \mathcal{G}_{0}}d_{A \cup A''}^{-}(X) + |E(\mathcal{F}_{1}, \mathcal{F}_{2})|  ~~~~~~~~~~~~~~~~~~~~~~~~~ (\text{since } \mathcal{G}_{0}=\mathcal{F}_{1} \uplus \mathcal{F}_{2}).  
\end{split}
\end{equation}


By Claim~\ref{fact} $(ii)$, $\cup \mathcal{F}_{3}=(\cup \mathcal{F}_{1}) \cup (\cup \mathcal{F}_{2})$.   
For every $e=uv \in E$ such that $u \notin (\cup \mathcal{F}_{1}) \cup  (\cup \mathcal{F}_{2})$ and $v \in (\cup \mathcal{F}_{1}) \cup (\cup \mathcal{F}_{2})$, by the definition of $A''$, $e$ is oriented from $u$ to $v$. Therefore  $e_{E}(\mathcal{F}_{3})= \sum_{X \in \mathcal{F}_{3}} d_{ A''}^{-}(X)$.  

By Proposition~\ref{17-P}, $\cup \mathcal{F}_{4}=(\cup \mathcal{F}_{1}) \cap (\cup \mathcal{F}_{2})$. 
For every $e=uv \in E$ such that $u \notin (\cup \mathcal{F}_{1}) \cap  (\cup \mathcal{F}_{2})$ and $v \in (\cup \mathcal{F}_{1}) \cap (\cup \mathcal{F}_{2})$, by the definition of $A''$, $e$ is oriented from $u$ to $v$. 
Therefore  $e_{E}(\mathcal{F}_{4})= \sum_{X \in \mathcal{F}_{4}} d_{A''}^{-}(X)$.  
Hence,
\begin{equation}\label{right}
\begin{split}
&e_{E}(\mathcal{F}_{3})+ \sum_{X \in \mathcal{F}_{3}}d_{A}^{-}(X) + e_{E}(\mathcal{F}_{4})+ \sum_{X \in \mathcal{F}_{4}}d_{A}^{-}(X)\\
= ~ & \sum_{X \in \mathcal{F}_{3}}d_{A''}^{-}(X)+ \sum_{X \in \mathcal{F}_{3}}d_{A}^{-}(X) + \sum_{X \in \mathcal{F}_{4}}d_{A''}^{-}(X)+ \sum_{X \in \mathcal{F}_{4}}d_{A}^{-}(X)\\
= ~  & \sum_{X \in \mathcal{F}_{3}} d_{A \cup A''}^{-}(X)+ \sum_{X \in \mathcal{F}_{4}} d_{A \cup A''}^{-}(X) \\
 = ~ & \sum_{X \in \mathcal{G}_{n}} d_{A \cup A''}^{-}(X)  ~~~~~~~~~~~~~~~~~~~~~~~~~ (\text{since } \mathcal{G}_{n}=\mathcal{F}_{3} \uplus \mathcal{F}_{4}).   
\end{split}
\end{equation}

In the process of PIEOs, for $i \in [n]$, suppose we obtain $\mathcal{G}_{i}$ by replacing a properly intersecting pair $X$ and $Y$ in $\mathcal{G}_{i-1}$ with $ X \cup Y$ and $X \cap Y$. 
Then $\mathcal{G}_{i-1} \setminus \{X, Y \}= \mathcal{G}_{i} \setminus \{X \cup Y, X \cap Y \}$.  
Since $d_{A \cup A''}^{-}$ is submodular on $2^{V}$, $d_{A \cup A''}^{-}(X)+ d_{A \cup A''}^{-}(Y) \geq d_{A \cup A''}^{-}(X\cup Y )+ d_{A \cup A''}^{-}(X \cap Y) $.   
Therefore $ \sum_{X \in \mathcal{G}_{i-1}} d_{A \cup A''}^{-}(X)  \geq \sum_{X \in \mathcal{G}_{i}} d_{A \cup A''}^{-}(X)$. 
It follows that 
\begin{equation}\label{G0Gn}
 \sum_{X \in \mathcal{G}_{0}} d_{A \cup A''}^{-}(X) \geq \sum_{X \in \mathcal{G}_{1}} d_{A \cup A''}^{-}(X) \geq \ldots \geq  \sum_{X \in \mathcal{G}_{n}} d_{A \cup A''}^{-}(X).
\end{equation}

Hence, we have
\[
\begin{split}
& e_{E}(\mathcal{F}_{1})+ \sum_{X \in \mathcal{F}_{1}}d_{A}^{-}(X) + e_{E}(\mathcal{F}_{2})+ \sum_{X \in \mathcal{F}_{2}}d_{A}^{-}(X) \\
= ~ & \sum_{X \in \mathcal{G}_{0}}d_{A \cup A''}^{-}(X) + |E(\mathcal{F}_{1}, \mathcal{F}_{2})| ~~~~~~~~~~~~~~~~~~~~~~~~~~~~~~~~~~~~~~~~~~ (\text{by (\ref{left})}) \\ 
\geq ~ &  \sum_{X \in \mathcal{G}_{n}}d_{A \cup A''}^{-}(X) + |E(\mathcal{F}_{1}, \mathcal{F}_{2})| ~~~~~~~~~~~~~~~~~~~~~~~~~~~~~~~~~~~~~~~~~~ (\text{by (\ref{G0Gn})}) \\
= ~ & e_{E}(\mathcal{F}_{3})+ \sum_{X \in \mathcal{F}_{3}}d_{A}^{-}(X)+ e_{E}(\mathcal{F}_{4})+ \sum_{X \in \mathcal{F}_{4}}d_{A}^{-}(X) + |E(\mathcal{F}_{1}, \mathcal{F}_{2})| ~~~ (\text{by (\ref{right})}).
\end{split}
\] 
\end{proof} 

The following lemma will be used in the final step to explain why we can orient an edge $e \in E$ to take care of all these critical subpartitions in $\mathcal{E}^{1}$ and $\mathcal{E}^{2}$. 

\begin{lemma}\label{emptyset} 
	For $\mathcal{F}_{1}, \mathcal{F}_{2} \in \mathcal{E}^{1} \cup \mathcal{E}^{2}$, we have
$E(\mathcal{F}_{1}, \mathcal{F}_{2}) = \emptyset$.
\end{lemma}

\begin{proof}
Suppose to the contrary that $E(\mathcal{F}_{1}, \mathcal{F}_{2}) \neq \emptyset$. Then, by Claim~\ref{>=|E|}, we have
\begin{equation}\label{>}
\begin{split}
& e_{E}(\mathcal{F}_{1})+ \sum_{X \in \mathcal{F}_{1}}d_{A}^{-}(X) + e_{E}(\mathcal{F}_{2})+ \sum_{X \in \mathcal{F}_{2}}d_{A}^{-}(X) \\
> ~ & e_{E}(\mathcal{F}_{3})+ \sum_{X \in \mathcal{F}_{3}}d_{A}^{-}(X)+ e_{E}(\mathcal{F}_{4})+ \sum_{X \in \mathcal{F}_{4}}d_{A}^{-}(X).
 \end{split}
\end{equation}

\noindent \textbf{Case 1:} Assume $\mathcal{F}_{1}, \mathcal{F}_{2} \in \mathcal{E}^{1}$.

By Claim~\ref{fact} $(ii)$, $V \setminus \cup \mathcal{F}_{3}$  $ = (V \setminus \cup \mathcal{F}_{1}) \cap (V \setminus \cup \mathcal{F}_{2})$.   By Proposition~\ref{17-P}, $V \setminus \cup \mathcal{F}_{4}$ $= (V \setminus \cup \mathcal{F}_{1}) \cup (V \setminus \cup \mathcal{F}_{2})$. 
Thus, 
\begin{equation}\label{F1F2toF3F4}
  \widetilde{f}(V \setminus \cup \mathcal{F}_{3}) + \widetilde{f}(V \setminus \cup \mathcal{F}_{4})= \widetilde{f}(V \setminus \cup \mathcal{F}_{1}) + \widetilde{f}(V \setminus \cup \mathcal{F}_{2}).
\end{equation}
Hence,
\begin{align*}
& e_{E}(\mathcal{F}_{3})+ \sum_{X \in \mathcal{F}_{3}}d_{A}^{-}(X) + e_{E}(\mathcal{F}_{4})+ \sum_{X \in \mathcal{F}_{4}}d_{A}^{-}(X)\\
\geq ~ &  k(|\mathcal{F}_{3}|-1)+ \widetilde{f}(V \setminus \cup \mathcal{F}_{3}) + k(|\mathcal{F}_{4}|-1)+ \widetilde{f}(V \setminus \cup \mathcal{F}_{4}) ~~~ (\text{by } (\ref{ine-f}))\\
 = ~ & k(|\mathcal{F}_{1}|-1)+ \widetilde{f}(V \setminus \cup \mathcal{F}_{1}) + k(|\mathcal{F}_{2}|-1)+ \widetilde{f}(V \setminus \cup \mathcal{F}_{2}) ~~~ (\text{by } (\ref{F1F2toF3F4}) \text{ and  Claim} ~ (\ref{fact}) (i))\\ 
 = ~ & e_{E}(\mathcal{F}_{1})+ \sum_{X \in \mathcal{F}_{1}}d_{A}^{-}(X) + e_{E}(\mathcal{F}_{2})+ \sum_{X \in \mathcal{F}_{2}}d_{A}^{-}(X)  ~~~~~~~~~~~~~~ (\text{since } \mathcal{F}_{1}, \mathcal{F}_{2} \in \mathcal{E}^{1}), 
\end{align*} 
but this is a contradiction to (\ref{>}).  

\smallskip

\noindent \textbf{Case 2:} Assume $\mathcal{F}_{1}, \mathcal{F}_{2} \in \mathcal{E}^{2}$.

The proof will use the function $g$  and the assumption (\ref{ine-g}),  also the definition of $\mathcal{E}^{2}$. The process is similar to Case 1,  details are skipped here.    

\smallskip

\noindent \textbf{Case 3:}  Assume $\mathcal{F}_{1} \in  \mathcal{E}^{1}$ and $\mathcal{F}_{2} \in \mathcal{E}^{2}$. 

By Claim~\ref{fact} (ii), $ V \setminus \cup \mathcal{F}_{3} \subseteq V \setminus \cup \mathcal{F}_{1}$, and $(V \setminus \cup \mathcal{F}_{1}) \setminus (V \setminus \cup \mathcal{F}_{3})=\cup \mathcal{F}_{3} \setminus \cup \mathcal{F}_{1}= ((\cup \mathcal{F}_{1}) \cup (\cup \mathcal{F}_{2}) ) \setminus \cup \mathcal{F}_{1}=\cup \mathcal{F}_{2} \setminus \cup \mathcal{F}_{1} $.  
By Proposition~\ref{17-P}, $\cup \mathcal{F}_{4} \subseteq \cup \mathcal{F}_{2} $, and $\cup \mathcal{F}_{2} \setminus \cup \mathcal{F}_{4}=  \cup \mathcal{F}_{2} \setminus ( (\cup \mathcal{F}_{1}) \cap (\cup \mathcal{F}_{2}))=\cup \mathcal{F}_{2} \setminus \cup \mathcal{F}_{1} $. 
Thus $(V \setminus \cup \mathcal{F}_{1}) \setminus (V \setminus \cup \mathcal{F}_{3})=\cup \mathcal{F}_{2} \setminus \cup \mathcal{F}_{4}$.  

By the assumption $f\leq g$, using $V \setminus \cup \mathcal{F}_{3} \subseteq V \setminus \cup \mathcal{F}_{1}$ and $\cup \mathcal{F}_{4} \subseteq \cup \mathcal{F}_{2} $, we have  
\begin{equation} \label{f<=g}
\begin{split}
 & \widetilde{f}(V \setminus \cup \mathcal{F}_{1})-\widetilde{f}(V \setminus \cup \mathcal{F}_{3})  =  \widetilde{f}((V \setminus \cup \mathcal{F}_{1}) \setminus (V \setminus \cup \mathcal{F}_{3})) ~~ \\
 \leq ~ & \widetilde{g}((V \setminus \cup \mathcal{F}_{1}) \setminus (V \setminus \cup \mathcal{F}_{3})) =  \widetilde{g}(\cup \mathcal{F}_{2} \setminus \cup \mathcal{F}_{4}) = \widetilde{g}(\cup \mathcal{F}_{2})- \widetilde{g}(\cup \mathcal{F}_{4}).
\end{split}
\end{equation}

Hence,
\begin{align*}
& e_{E}(\mathcal{F}_{3})+ \sum_{X \in \mathcal{F}_{3}}d_{A}^{-}(X) + e_{E}(\mathcal{F}_{4})+ \sum_{X \in \mathcal{F}_{4}}d_{A}^{-}(X)\\
\geq ~ &  k(|\mathcal{F}_{3}|-1)+ \widetilde{f}(V \setminus \cup \mathcal{F}_{3}) + k|\mathcal{F}_{4}|- \widetilde{g}( \cup \mathcal{F}_{4})    ~~~~~~~~~~ \text{~(by (\ref{ine-f}) and (\ref{ine-g}))}\\
 \geq ~ & k(|\mathcal{F}_{1}|-1)+ \widetilde{f}(V \setminus \cup \mathcal{F}_{1}) + k|\mathcal{F}_{2}|- \widetilde{g}(\cup \mathcal{F}_{2})  ~~~~~~~~~~ \text{~(by (\ref{f<=g}) and Claim~\ref{fact} (i))}\\
= ~ & e_{E}(\mathcal{F}_{1})+ \sum_{X \in \mathcal{F}_{1}}d_{A}^{-}(X) + e_{E}(\mathcal{F}_{2})+ \sum_{X \in \mathcal{F}_{2}}d_{A}^{-}(X)  ~~~~~~~~ \text{~(since $\mathcal{F}_{1} \in \mathcal{E}^{1}$, $\mathcal{F}_{2} \in \mathcal{E}^{2}$)},  
\end{align*}
but this is a contradiction to (\ref{>}). This proves the lemma. 
\end{proof}

To finish the proof, we pick an edge $e_{0} \in E$, orient $e_0$ to  $\vec{e}_0$ as following:  
If there exists an $\mathcal{F}_{0} \in \mathcal{E}^{1} \cup \mathcal{E}^{2}$ such that one end of $e_{0}$, say $v \in \cup \mathcal{F}_{0}$ and the other end $u \notin \cup \mathcal{F}_{0}$, then we orient $e_{0}$ from $u$ to $v$ (i.e., $\vec{e}_0 := \overrightarrow{uv}$); otherwise, orient $e_{0}$ to $\vec{e}_0$ arbitrarily. 
Then we do  $A:= A+\vec{e}_0$, $E:=E-e_0$, $F':=(V; A, E)$. 
It suffices for us to prove that  for any $\mathcal{F} \in \mathcal{D}(V)$, assumptions (\ref{ine-f}) and (\ref{ine-g}) still hold for this new mixed graph $F'$.      

Note that the subpartitions in $\mathcal{E}^{1}$ and $\mathcal{E}^{2}$ are the ones that make assumptions (\ref{ine-f}) or (\ref{ine-g}) tight in the mixed graph $F$. 
If $\mathcal{F} \notin \mathcal{E}^{1} \cup \mathcal{E}^{2}$, since $e_{E}(\mathcal{F})+ \sum_{X \in \mathcal{F}}d_{A}^{-}(X)$ is decreased  by at most $1$, (\ref{ine-f}) and (\ref{ine-g}) still hold in $F'$. 
Otherwise, we have $\mathcal{F} \in \mathcal{E}^{1} \cup \mathcal{E}^{2}$. Then we prove (next)  that $e_{E}(\mathcal{F})+ \sum_{X \in \mathcal{F}}d_{A}^{-}(X)$ keeps the same in $F$ and $F'$.
 Thus (\ref{ine-f}) and (\ref{ine-g}) still hold in $F'$.   

Suppose $\vec{e}_0 = \overrightarrow{uv}$ in $F'$.   
If both $u, v \in X$ for some $X \in \mathcal{F}$, or both $u, v \notin \cup \mathcal{F}$, then $e_{E}(\mathcal{F})$ and $\sum_{X \in \mathcal{F}} d_{A}^{-}(X)$ keep the same in $F$ and $F'$.   

If, for some $X, Y \in \mathcal{F}$ and $X\neq Y$, $v \in X$ and $u \in Y $,  then $e_{E}(\mathcal{F})$ is decreased by $1$ and $\sum_{X \in \mathcal{F}} d_{A}^{-}(X)$ is increased by $1$ in $F'$. Therefore  $e_{E}(\mathcal{F})+ \sum_{X \in \mathcal{F}}d_{A}^{-}(X)$ keeps the same in $F$ and $F'$. 

The left cases are either (A) $v \in \cup \mathcal{F}$ and $u \notin \cup \mathcal{F}$; or (B) $u \in \cup \mathcal{F}$ and $v \notin \cup \mathcal{F}$. But Case (B) can not happen.    
Indeed, assume to the contrary that $u \in \cup \mathcal{F}$ and $v \notin \cup \mathcal{F}$. Since $\vec{e}_0$ is oriented from $u$ to $v$,  there exists an  $\mathcal{F}_{0} \in \mathcal{E}^{1} \cup \mathcal{E}^{2}$ such that $v \in \cup \mathcal{F}_{0}$ and $u \notin \cup \mathcal{F}_{0}$. we conclude that $e_{0} \in E(\mathcal{F}_{0}, \mathcal{F})$. But by Lemma~\ref{emptyset}, $E(\mathcal{F}_{0}, \mathcal{F})=\emptyset$. This is a contradiction.   
So the only left case is (A) $v \in \cup \mathcal{F}$ and $u \notin \cup \mathcal{F}$. Then $e_{E}(\mathcal{F})$ is decreased by $1$ and $\sum_{X \in \mathcal{F}} d_{A}^{-}(X)$ is increased by $1$ in $F'$. Therefore  $e_{E}(\mathcal{F})+ \sum_{X \in \mathcal{F}}d_{A}^{-}(X)$ keeps the same in $F$ and $F'$. This finishes the proof of Theorem \ref{mix-spanning}.


\begin{thebibliography}{99} 
	
\bibitem{berczi-frank1} K. B\'{e}rczi, A. Frank, Supermodularity in unweighted graph optimization I: Branchings and matchings, Math. Oper. Res. 43(3) (2018), 726-753. 

\bibitem{C-6}
	 M.-C. Cai,  Arc-disjoint arborescences of digraphs, J. Graph Theory 7(2) (1983), 235-240.
	
\bibitem{durand} O. Durand de Gevigney, V.-H. Nguyen, and Z. Szigeti, Matroid-based packing of arborescences, SIAM J. Discrete Math., 27 (2013), 567-574.


\bibitem{edmonds} J. Edmonds,  Edge-disjoint branchings, Combinatorial algorithms (Courant Comput. Sci. Sympos. 9, New York Univ., New York, 1972), pp. 91--96. Algorithmics Press, New York, 1973. 

\bibitem{fortier} Q. Fortier, Cs. Kir\'{a}ly, M. L\'{e}onard, Z. Szigeti, A. Talon, Old and new results on packing arborescences in directed hypergraphs, Discrete Appl. Math. 242 (2018), 26-33.


\bibitem{F-13}
	A. Frank, On disjoint trees and arborescences,  Algebraic Methods in Graph Theory, Colloquia Mathematica Soc. J. Bolyai, Vol. 25
	(North-Holland, Amsterdam)  (1978), 159-169.

\bibitem{F-2011}	
	A. Frank, Connections in Combinatorial Optimization, Oxford Lecture Series in Mathematics and Its Applications, Vol. 38 (Oxford
	University Press, Oxford, UK)  (2011).


\bibitem{GY-1} H. Gao, D. Yang,
Packing branchings under cardinality constraints on their root sets, arXiv:1908.10795v2 [math.CO] 9 Feb 2020 (to appear in European J. Combin.).  

	
\bibitem{GY-max-mix} H. Gao, D. Yang,
Packing of maximal independent mixed arborescences, arXiv:2003.04062v1 [math.CO] 9 Mar 2020. 

\bibitem{kamiyama} N. Kamiyama, N. Katoh, and A. Takizawa, Arc-disjoint in-trees in directed graphs, Combinatorica, 29 (2009), 197-214.

\bibitem{kiraly} Cs. Kir\'{a}ly, On maximal independent arborescence packing, SIAM J. Discrete. Math. 30 (4) (2016), 2107-2114.

\bibitem{matsuoka} T. Matsuoka, S. Tanigawa, On reachability mixed arborescence packing, Discrete Optimization 32 (2019), 1-10.

\bibitem{NW61} C. St. J. A. Nash-Williams, Edge-disjoint spanning trees of
finite graphs. J. Lond. Math. Soc. 36 (1961), 445--450.

\bibitem{T} W. T. Tutte, On the problem of decomposing a graph into $n$
connected factors. J. Lond. Math. Soc. 36 (1961), 221--230.
\end{thebibliography}
\end{document}